\documentclass[12pt,english]{article}
\usepackage[T1]{fontenc}
\usepackage[latin9]{inputenc}
\usepackage{amsthm}
\usepackage{amsmath}
\usepackage{amssymb}
\usepackage{esint}

\makeatletter
\theoremstyle{plain}
\newtheorem{thm}{\protect\theoremname}
  \theoremstyle{remark}
  \newtheorem{rem}[thm]{\protect\remarkname}
  \theoremstyle{definition}
  \newtheorem{defn}[thm]{\protect\definitionname}
  \theoremstyle{plain}
  \newtheorem{lem}[thm]{\protect\lemmaname}
  \theoremstyle{plain}
  \newtheorem{prop}[thm]{\protect\propositionname}
  \theoremstyle{plain}
  \newtheorem{cor}[thm]{\protect\corollaryname}

\usepackage{pdfsync,hyperref}
\DeclareMathAlphabet{\mathitbf}{OML}{cmm}{b}{it}

\makeatother

\usepackage{babel}
  \providecommand{\corollaryname}{Corollary}
  \providecommand{\definitionname}{Definition}
  \providecommand{\lemmaname}{Lemma}
  \providecommand{\propositionname}{Proposition}
  \providecommand{\remarkname}{Remark}
\providecommand{\theoremname}{Theorem}

\begin{document}

\title{Grey Brownian motion local time: Existence and weak-approximation}

\author{\textbf{Jos{\'e} Lu{\'\i}s da Silva},\\
CCM, University of Madeira, Campus da Penteada,\\
9020-105 Funchal, Portugal.\\
Email: luis@uma.pt\and \textbf{Mohamed Erraoui}\\
Université Cadi Ayyad, Facult{\'e} des Sciences Semlalia,\\
D{\'e}partement de Math{\'e}matiques, BP 2390, Marrakech, Maroc\\
Email: erraoui@uca.ma}
\date{}
\maketitle
\begin{abstract}
In this paper we investigate the class of grey Brownian motions $B_{\alpha,\beta}$
($0<\alpha<2$, $0<\beta\leq1$). We show that grey Brownian motion
admits different representations in terms of certain known processes,
such as fractional Brownian motion, multivariate elliptical distribution
or as a subordination. The weak convergence of the increments of $B_{\alpha,\beta}$
in $t$, $w$-variables are studied. Using the Berman criterium we
show that $B_{\alpha,\beta}$ admits a $\lambda$-square integrable
local time $L^{B_{\alpha,\beta}}(\cdot,I)$ almost surely ($\lambda$
Lebesgue measure). Moreover, we prove that this local time can be
weak-approximated by the number of crossings $C^{B_{\alpha,\beta}^{\varepsilon}}(x,I)$,
of level $x$, of the convolution approximation $B_{\alpha,\beta}^{\varepsilon}$
of grey Brownian motion. \medskip{}

\end{abstract}
\tableofcontents{}

\section{Introduction}

Grey Brownian motion was introduced by W.~Schneider in \cite{MR1190506, Schneider90}
as a stochastic model for slow-anomalous diffusion described by the
time-fractional diffusion equation. Later F.~Minardi, A.~Mura and
G.~Pagnini \cite{Mura_mainardi_09,Mura_Pagnini_08}, extended this
class, so called ``generalized'' grey Brownian motion which includes
stochastic models for slow and fast-anomalous diffusion, i.e., the
time evolution of the marginal density function is described by partial
integro-differential equations of fractional type, cf.~Proposition~4.1
in \cite{Mura_mainardi_09}. In this paper we will consider this extended
class and study different representations, prove the existence of
local time and show the weak-convergence of the increments in the
$t$, $w$-variables.

To describe these results more precisely, let us recall the definition
of ``generalized'' grey Brownian motion $B_{\alpha,\beta}$ (denoted
by ggBm in Mura et al.~\cite{Mura_mainardi_09}) and from now denoted
simply by gBm for short. We will not reproduce the all construction
of the gBm here, the interested reader can find it in \cite{Mura_mainardi_09}
and references therein. The grey noise space is the probability space
$(S'(\mathbb{R}),\mathcal{B}(S'(\mathbb{R})),\mu_{\alpha,\beta})$,
where $S'(\mathbb{R})$ is the space of tempered distributions defined
on $\mathbb{R}$, $\mathcal{B}(S'(\mathbb{R}))$ is the $\sigma$-algebra
generated by the cylinder sets and $\mu_{\alpha,\beta}$ is the grey
noise measure given by its characteristic functional
\begin{equation}
\int_{S'(\mathbb{R})}e^{i\langle w,\varphi\rangle}\, d\mu_{\alpha,\beta}(w)=E_{\beta}\left(-\frac{1}{2}
\|\varphi\|_{\alpha}^{2}\right),\;\varphi\in S(\mathbb{R}),\;0<\beta\leq1,\;0<\alpha<2.\label{eq:gnm}
\end{equation}
Here $\langle\cdot,\cdot\rangle$ is the canonical bilinear pairing
between $S(\mathbb{R})$ and $S'(\mathbb{R})$ and $E_{\beta}$ is
the Mittag-Leffler function of order $\beta$, defined by
\[
E_{\beta}(x)=\sum_{n=0}^{\infty}\frac{x^{n}}{\Gamma(\beta n+1)},\quad x\in\mathbb{R}.
\]
In addition, 
\[
\|\varphi\|_{\alpha}^{2}=\Gamma(1+\alpha)\sin\left(\frac{\pi\alpha}{2}\right)\int_{\mathbb{R}}|x|^{1-\alpha}|
\tilde{\varphi}(x)|^{2}\, dx,\quad\varphi\in S(\mathbb{R}),
\]
with $\tilde{\varphi}$ denotes the Fourier transform of $\varphi$.
The range $0<\beta\leq1$ is to ensure the complete monotonicity of
$E_{\beta}$ (needed in (\ref{eq:gnm})) while the range $0<\alpha<2$
is chosen such that $\|1\!\!1_{[a,b)}\|_{\alpha}^{2}=(b-a)^{\alpha}<\infty$. 

It is easy to show that the random variable $X_{\alpha,\beta}(1\!\!1_{[0,t)})(\cdot)=
\langle\cdot,1\!\!1_{[0,1)}\rangle$,
$t\geq0$ is well defined as an element in $L^{2}(S'(\mathbb{R}),\mathcal{B}(S'(\mathbb{R})),
\mu_{\alpha,\beta})=:L^{2}(\mu_{\alpha,\beta})$
and 
\[
\|X_{\alpha,\beta}(1\!\!1_{[0,t)})\|_{L^{2}(\mu_{\alpha,\beta})}^{2}=\frac{1}{\Gamma(\beta+1)}\|1\!\!1_{[0,t)}
\|_{\alpha}^{2}=\frac{1}{\Gamma(\beta+1)}t^{\alpha}.
\]
The gBm $B_{\alpha,\beta}$ is then defined as the stochastic process
\[
B_{\alpha,\beta}=\big\{ B_{\alpha,\beta}(t):=X_{\alpha,\beta}(1\!\!1_{[0,t)}),\quad t\geq0\big\}.
\]
The following properties can be easily derived from (\ref{eq:gnm})
and the fact that $\|1\!\!1_{[0,t)}\|_{\alpha}^{2}=t^{\alpha}$.
\begin{enumerate}
\item $B_{\alpha,\beta}(0)=0$ almost surely. In addition, for each $t\geq0$,
the moments of any order are given by
\[
\begin{cases}
\mathbb{E}(B_{\alpha,\beta}^{2n+1}(t)) & =0,\\
\noalign{\vskip4pt}\mathbb{E}(B_{\alpha,\beta}^{2n}(t)) & =\frac{(2n)!}{2^{n}\Gamma(\beta n+1)}t^{n\alpha}.
\end{cases}
\]
Here $\mathbb{E}$ denotes the expectation with respect to $\mu_{\alpha,\beta}$. 
\item For each $t,s\geq0$, the characteristic function of the increments
is 
\begin{equation}
\mathbb{E}\big(e^{i\theta(B_{\alpha,\beta}(t)-B_{\alpha,\beta}(s))}\big)=E_{\beta}\left(-\frac{
\theta^{2}}{2}|t-s|^{\alpha}\right),\quad\theta\in\mathbb{R}.\label{eq:cf_gBm_increments}
\end{equation}

\item The covariance function has the form
\begin{equation}
\mathbb{E}(B_{\alpha,\beta}(t)B_{\alpha,\beta}(s))=\frac{1}{2\Gamma(\beta+1)}\big(t^{\alpha}+s^{\alpha}-|t-s|^
{\alpha}\big),\quad t,s\geq0.\label{eq:auto-cv-gBm}
\end{equation}

\end{enumerate}
All these properties may be summarized as follows. For any $0<\alpha<2$
and $0<\beta\leq1$, the gBm $B_{\alpha,\beta}(t)$, $t\geq0$, is
$\frac{\alpha}{2}$-self-similar with stationary increments. We refer
to \cite{Mura_mainardi_09} for the proof and more details. This class
includes fractional Brownian motion (fBm) for $\beta=1$, and Brownian
motion (Bm) for $\alpha=\beta=1$. The $B_{\alpha,\beta}(t)$ marginal
density function $f_{\alpha,\beta}(t)$ is given by 
\[
f_{\alpha,\beta}(x,t)=\frac{t^{-\alpha/2}}{\sqrt{2}}M_{\beta/2}(\sqrt{2}|x|t^{-\alpha/2}),
\]
where $M_{\beta}$ is the so-called $M$-Wright probability density
function (a natural generalization of the Gaussian density) with Laplace
transform
\begin{equation}
\int_{0}^{\infty}e^{-s\tau}M_{\beta}(\tau)\, d\tau=E_{\beta}(-s)\label{eq:M_wright}
\end{equation}
and moments given by
\begin{equation}
\int_{0}^{\infty}\tau^{k}M_{\beta}(\tau)\, d\tau=\frac{\Gamma(k+1)}{\Gamma(\beta k+1)},
\quad\forall k\in\mathbb{N}.\label{moments}
\end{equation}
The function $f_{\alpha,\beta}$ is the fundamental solution of the
stretched time-fractional diffusion equation, cf.~\cite[eq.~(6)]{Mura_Pagnini_08},
with initial condition $f_{\alpha,\beta}(x,0)=\delta(x)$. The Fourier
transform of $f_{\alpha,\beta}$ is given by
\[
\tilde{f}_{\alpha,\beta}(\theta,t)=\mathbb{E}(e^{i\theta B_{\alpha,\beta}(t)})=E_{\beta}\left(-\frac{1}{2}
\|\theta1\!\!1_{[0,t)}\|_{\alpha}^{2}\right)=E_{\beta}\left(-\frac{\theta^{2}}{2}t^{\alpha}\right),\quad\theta\in\mathbb{R}.
\]
The multidimensional case, given $\theta=(\theta_{1},\ldots,\theta_{n})\in\mathbb{R}^{n}$,
$n\in\mathbb{N}$ and any collection $\{B_{\alpha,\beta,}(t_{1}),\ldots,B_{\alpha,\beta,}(t_{n})\}$
with $0\leq t_{1}<t_{2}<\ldots<t_{n}<\infty$ we have
\begin{equation}
\tilde{f}(\theta,t)=\mathbb{E}\left(\exp\left(i\sum_{k=1}^{n}\theta_{k}B_{\alpha,\beta}(t_{k})\right)\right)=
E_{\beta}\left(-\frac{1}{2}\theta^{\top}\Sigma_{\alpha}\theta\right),\label{eq:gBm_nG}
\end{equation}
where $\Sigma_{\alpha}=(a_{i,j})_{i,j=1}^{n}$ is the matrix given
by 
\[
a_{i,j}=t_{i}^{\alpha}+t_{j}^{\alpha}-|t_{i}-t_{j}|^{\alpha}.
\]
Equation (\ref{eq:gBm_nG}) shows that gBm, which is not Gaussian
in general, is a stochastic process defined only through its first
and second moments which is a property of Gaussian processes indeed.

Till now we have described the class of processes which we are interested
in, namely gBm $B_{\alpha,\beta}$. In Section~\ref{sec:repr} we
show that gBm may be represented in terms of certain known processes
using its finite dimensional structure, Subsection~\ref{sub:fdr}.
In particular, the representation using fBm, see (\ref{gbm-rep})
below, is very useful in order to study gBm. While in Subsection~\ref{sub:odr}
we obtain other representations of gBm as subordinations valid for
one-dimensional distributions, namely Bm and fBm subordinated to certain
increasing processes. In Section~\ref{sec:wc-gBm} we study the weak
convergence of the increments of gBm using the representation of gBm
in terms of fBm. All the results on that section are strongly associated
to those for fBm despite the gBm being more general. Finally, on Section~\ref{sec:lt-gBm}
we make the first studies concerning the local times of gBm $L^{B_{\alpha,\beta}}$,
namely its existence using a criteria due to Berman \cite{Berman:1969wk}
and the weak approximation of the gBm local times by the number of
crossings. Thus, we have shown that $L^{B_{\alpha,\beta}}$ is $\lambda$-square
integrable, almost surely, cf.~Theorem~\ref{thm:gBm_LT}, and as
a result the occupation formula is valid. Moreover, defining the regularization
of $B_{\alpha,\beta}$ by convolution as $B_{\alpha,\beta}^{\varepsilon}$,
i.e., $B_{\alpha,\beta}^{\varepsilon}:=\psi_{\varepsilon}*B_{\alpha,\beta}$
for a proper $\psi$, we obtain a weak approximation of the local
times $L^{B_{\alpha,\beta}}$ by the number of crossings of $B_{\alpha,\beta}^{\varepsilon}$.
The proof is based on the Banach-Kac formula (see \cite[pag.~253]{Kratz2006})
and the occupation formula. For the rest of the paper, we fix the
following notation: by $N$ we denote the standard Gaussian random
variable and by $F_{\sigma}$ the centered Gaussian distribution with
variance $\sigma$ corresponding to the random variable $N_{\sigma}$.

\section{Representations of grey Brownian motion}

\label{sec:repr}

In this section we will show that gBm $B_{\alpha,\beta}$ admits different
representations which will be useful in proving certain properties
on the next sections. This is related to the fact that these representations
involves certain known processes, such as fractional Brownian motion
(fBm), and some questions related to gBm reduces to those of fBm or
others. As an example, the Hölder continuity of the trajectories of
gBm immediately follows from those of fBm, see (\ref{eq:cont_gBm})
below.

\subsection{Finite dimensional representation}

\label{sub:fdr}

It was shown in \cite{Mura_Pagnini_08} that the gBm $B_{\alpha,\beta}$
admits the following representation
\begin{equation}
\big\{ B_{\alpha,\beta}(t),\; t\geq0\big\}\overset{d}{=}\big\{\sqrt{Y_{\beta}}B_{H}(t),\; t\geq0\big\},\label{gbm-rep}
\end{equation}
where $\overset{d}{=}$ denotes the equality of the finite dimensional
distribution and $B_{H}$ is a standard fBm with Hurst parameter $H=\alpha/2$.
$Y_{\beta}$ is an independent non-negative random variable with probability
density function $M_{\beta}(\tau)$, $\tau\geq0$.
\begin{rem}
\begin{enumerate}\item A process with the representation given as
in (\ref{gbm-rep}) is known to be variance mixture of normal distributions,
see \cite[Chap.~VI]{Steutel_van_Harn_2004} for details. 

\item It follows from the representation (\ref{gbm-rep}) that the
Hölder continuity of the trajectories of gBm reduces to the Hölder
continuity of the fBm. Thus we have
\begin{equation}
\mathbb{E}(|B_{\alpha,\beta}(t)-B_{\alpha,\beta}(s)|^{p})=c_{p}|t-s|^{p\alpha/2}.\label{eq:cont_gBm}
\end{equation}
\item The representation also provides a way to study and simulate
the trajectories of gBm once we know how to generate the random variable
$Y_{\beta}$. 

\end{enumerate}
\end{rem}
The grey Brownian motion admits also other representations which we
present in the following. Before we give a definition.
\begin{defn}
[cf.~\cite{Fang90}]Consider a $n$-dimensional random vector $\mathitbf{X}=(X_{1},\ldots,X_{n})$.
The random vector $\mathitbf{X}$ has a multivariate elliptical distribution,
denoted by $E_{n}(\mu,\Sigma,\phi)$, if its characteristic function
can be expressed as
\[
\mathbb{E}(e^{i\theta\cdot\mathitbf{X}})=\exp\left(i\theta\cdot\mu\right)\phi\left(\frac{1}{2}
\theta^{\top}\Sigma\theta\right),\quad\theta\in\mathbb{R}^{n},
\]
where $\mu\in\mathbb{R}^{n}$, $\Sigma$ is a positive-definite matrix
and $\phi$, called the characteristic generator, is such that $\phi(\theta_{1}^{2}+\ldots+\theta_{n}^{2})$
is a $n$-dimensional characteristic function. If $\mathitbf{X}$
has a probability density function $f_{\mathitbf{X}}$, then it is
of the form
\[
f_{\mathitbf{X}}(x)=\frac{c}{\sqrt{\det(\Sigma)}}g\left(\frac{1}{2}(x-\mu)^{\top}\Sigma^{-1}(x-\mu)\right),
\]
for a certain function $g$, called density generator. In \cite{Fang90}
it is shown that 
\[
\int_{0}^{\infty}x^{n/2-1}g(x)\, dx<\infty,
\]
is sufficient for $g$ to be the density generator. 
\end{defn}
Let us consider, for a fixed $n\in\mathbb{N}$, the vector 
\[
\mathitbf{B}_{\alpha,\beta}:=\big(B_{\alpha,\beta}(t_{1}),\ldots,B_{\alpha,\beta}(t_{n})\big).
\]
Then $\mathitbf{B}_{\alpha,\beta}$ has a multivariate elliptical
distribution, written as 
\[
E_{n}(0,\Sigma_{\alpha},E_{\beta}(-\cdot)),
\]
The associated density generator of $\mathitbf{B}_{\alpha,\beta}$
is given by 
\[
g_{\beta}(x)=\frac{1}{(2\pi)^{n/2}}\int_{0}^{\infty}\tau^{-n/2}e^{-x/\tau}M_{\beta}(\tau)\, d\tau.
\]
A simple calculation shows that 
\[
\int_{0}^{\infty}x^{n/2-1}g_{\beta}(x)\, dx=(2\pi)^{-n/2}\Gamma\left(\frac{n}{2}\right)<\infty
\]
which guarantees $g_{\beta}$ to be the density generator of the vector
$\mathitbf{B}_{\alpha,\beta}$. Furthermore, $\mathitbf{B}_{\alpha,\beta}$
may be represented by 
\[
\mathitbf{B}_{\alpha,\beta}=R_{\beta}A_{\alpha}U,
\]
where $R_{\beta}\geq0$ is a radial random variable, $A_{\alpha}$
is such that $\Sigma_{\alpha}=A_{\alpha}A_{\alpha}^{\top}$ and $U$
is the uniform distribution on the sphere $\{x\in\mathbb{R}^{n}:\;\|x\|=1\}$.

\subsection{One dimensional representation}

\label{sub:odr}

On this subsection we obtain two representations of gBm as subordinations,
valid for one-dimensional distributions. At first we show that gBm
may be represented as a subordination of Brownian motion by a $\beta$-stable
subordinator, cf.~(\ref{eq:subord_Bm}). The second representation
of gBm as a subordination of fBm using as subordinator a process with
one-dimensional distribution related to the \emph{M}-Wright function. 

Let $S=\{S(t),\; t\in[0,1]\}$ be a $\beta$-stable subordinator and
define the inverse process of $S$ by
\[
E(x):=\inf\{t:\; S(t)>x\},\quad x\in\mathbb{R}^{+}.
\]
$E(x)$ is $\frac{1}{\beta}$-self-similar process with no independent/stationary
increments. Bingham \cite{Bingham1971} and Bondesson et al.~\cite{Bondesson1996}
showed that $E(x)$ has a Mittag-Leffler distribution
\[
\mathbb{E}(e^{-sE(x)})=E_{\beta}(-sx^{\beta}).
\]
It follows that
\[
\mathbb{E}(e^{-sE(x^{\alpha/\beta})})=E_{\beta}(-sx^{\alpha}).
\]
On the other hand, we have the equality in law $E(x)=(S(1)/x)^{-\beta}$
which implies 
\[
\mathbb{E}\big(e^{-sS^{-\beta}(1)}\big)=E_{\beta}(-s)
\]
and 
\[
\mathbb{E}\big(e^{-sS^{-\beta}(t^{-\alpha})}\big)=E_{\beta}(-st^{\alpha}).
\]
As a consequence, we obtain the following representation for the gBm
\begin{equation}
B_{\alpha,\beta}(t)=B(E(t^{\alpha/\beta}))=B(S^{-\beta}(t^{-\alpha})),\label{eq:subord_Bm}
\end{equation}
where $B$ is a standard Brownian motion independent of $S$ and the
equalities are valid only for one-dimensional distributions. 

Let $D_{\beta}=\{D_{\beta}(t),\; t\geq0\}$ be the process with one-dimensional
distribution given by, see \cite{Mainardi_Mura_Pagnini_2010},
\[
f_{D_{\beta}(t)}(x)=t^{-\beta}M_{\beta}(xt^{-\beta}),\quad x,t\geq0.
\]
The grey Brownian motion is represented as
\[
B_{\alpha,\beta}(t)=B_{H}\big(D_{\beta}^{1/\alpha}(t^{\alpha/\beta})\big)
\]
where $B_{H}$ and $D_{\beta}$ are independent. The equality is valid
only for one-dimensional distributions. The density $f_{D_{\beta}(t)}$
is the fundamental solution of the time-fractional drift equation
\[
\mathcal{D}_{t}^{\beta}f_{D_{\beta}(t)}(x)=-\frac{\partial}{\partial x}f_{D_{\beta}(t)}(x),
\]
where $\mathcal{D}_{t}^{\beta}$ denotes de Caputo derivative.

\section{Weak convergence of the increments of gBm}

\label{sec:wc-gBm}

In this section we are going to study the weak convergence of the
increments of gBm in the $t$, $w$-variables. This question is related
to the uniqueness for the moment problem of two distributions. More
precisely, if two distributions have the same moments are they the
same? The additional assumption required (see \cite[Corollary~2, pag.~296]{S96})
is satisfied in all the cases studied below.

\subsection{Weak convergence in the $t$-variable}

We are interested in the increments of gBm $B_{\alpha,\beta}$, namely
we define 
\[
Z_{\alpha,\beta,\varepsilon}(t):=\varepsilon^{-\alpha/2}\big(B_{\alpha,\beta}(t+\varepsilon)-
B_{\alpha,\beta}(t)\big),\quad\varepsilon>0,
\]
and would like to study the weak convergence of 
\[
\lambda\{t\in[0,1],\; Z_{\alpha,\beta,\varepsilon}(t)\leq x\}
\]
as $\varepsilon\rightarrow0$. This is equivalent to find the limit
of the $t$-characteristic function 
\[
\lim_{\varepsilon\rightarrow0}\int_{0}^{1}e^{iuZ_{\alpha,\beta,\varepsilon}(t)}\, dt.
\]
This question is related to the moment problem, that is, study the
limit 
\[
\lim_{\varepsilon\rightarrow0}\int_{0}^{1}Z_{\alpha,\beta,\varepsilon}^{k}(t)\, dt,\qquad k\in\mathbb{N}.
\]
From now on we use the representation (\ref{gbm-rep}) for the gBm,
$B_{\alpha,\beta}=\sqrt{Y_{\beta}}B_{H}$. 
\begin{lem}
\textup{For each $k\in\mathbb{N}$ and $t\in[0,1]$ we have 
\[
\mathbb{E}\left(\int_{0}^{t}Z_{\alpha,\beta,\varepsilon}^{k}(s)\, ds-tY_{\beta}^{k/2}
\mathbb{E}\big(N^{k}\big)\right)^{2}=O(\sqrt{\varepsilon}).
\]
}\end{lem}
\begin{proof}
Due to the independence of $Y_{\beta}$ and $B_{H}$ we have
\begin{eqnarray*}
 &  & \mathbb{E}\left(\int_{0}^{t}Z_{\alpha,\gamma,\varepsilon}^{k}(s)\, ds-tY_{\beta}^{k/2}
 \mathbb{E}\big(N^{k}\big)\right)^{2}\\
 & = & \mathbb{E}(Y_{\beta}^{k})\mathbb{E}\left(\int_{0}^{t}\varepsilon^{-kH}\big(B_{H}(s+\varepsilon)-
 B_{H}(s)\big)^{k}\, ds-t\mathbb{E}\big(N^{k}\big)\right)^{2}\\
 & = & \frac{\Gamma(k+1)}{\Gamma(\beta k+1)}\mathbb{E}\left(\int_{0}^{t}\varepsilon^{-kH}
 \big(B_{H}(s+\varepsilon)-B_{H}(s)\big)^{k}\, ds-t\mathbb{E}\big(N^{k}\big)\right)^{2}.
\end{eqnarray*}
In the last equality we used (\ref{moments}). It follows from the
proof of Theorem~2.1 in \cite{Azais:1996} that 
\[
\mathbb{E}\left(\int_{0}^{t}\varepsilon^{-kH}\big(B_{H}(s+\varepsilon)-B_{H}(s)\big)^{k}\, ds
-t\mathbb{E}\big(N^{k}\big)\right)^{2}=O(\sqrt{\varepsilon}),
\]
see also \cite{Nourdin-2003}. This implies the result of the lemma. 
\end{proof}
As a consequence of the Hölder continuity of the trajectories of gBm
$B_{\alpha,\beta}$, see (\ref{eq:cont_gBm}), and the Borel-Cantelli
lemma, we obtain the almost sure convergence, see \cite{Azais:1996}
or \cite{Nourdin-2003}.
\begin{prop}
We have almost surely 
\[
\int_{0}^{t}Z_{\alpha,\beta,\varepsilon}^{k}(s)\, ds\longrightarrow tY_{\beta}^{k/2}
\mathbb{E}\big(N^{k}\big),\;\varepsilon\rightarrow0,
\]
for any $t\in[0,1]$.
\end{prop}
The convergence of the moments implies the weak convergence of the
distribution $\lambda\big\{ t\in[0,1],\; Z_{\alpha,\beta,\varepsilon}(t)\leq x\big\}$,
this is sated in the following theorem. 
\begin{thm}
1. For almost surely for any $t\in[0,1]$ we have
\begin{equation}
\lambda\big\{ s\in[0,t],\; Z_{\alpha,\beta,\varepsilon}(s)\leq x\big\}\longrightarrow tF_{Y_{\beta}}(x),\;
\varepsilon\rightarrow0,\; x\neq0.\label{eq:conv-law1}
\end{equation}
2. For almost surely for each interval $I\subset\mathbb{R}_{+}$ and
all $x\in\mathbb{R}$, we have 
\begin{equation}
\lambda\big\{ t\in I,\; Z_{\alpha,\beta,\varepsilon}(t)\leq x\big\}\longrightarrow\lambda(I)
F_{Y_{\beta}}(x),\;\varepsilon\rightarrow0.\label{eq:conv-law2}
\end{equation}
This is obtained from 1.~using density argument.
\end{thm}
For non integer $k$ we have an analogous result based on a deep result
of Marcus and Rosen \cite{Marcus-Rosen08}.
\begin{cor}
It follows from Theorem 2.3 in \cite{Marcus-Rosen08} that for any
$1\leq p<\infty$ 
\begin{equation}
\int_{0}^{t}|Z_{\alpha,\beta,\varepsilon}(s)|^{p}\, ds\longrightarrow tY_{\beta}^{p/2}
\mathbb{E}\big(|N|^{p}\big),\;\varepsilon\rightarrow0\label{eq:moduli-cont}
\end{equation}
holds for any $t\in[0,1]$ almost surely.
\end{cor}
The above result is also valid in a more general context. More precisely,
let $B_{\alpha,\beta}^{\varepsilon}$ be the convolution approximation
of $B_{\alpha,\beta}$ (or regularization of $B_{\alpha,\beta}$)
defined by
\[
B_{\alpha,\beta}^{\varepsilon}=\psi_{\varepsilon}*B_{\alpha,\beta},\qquad\mathrm{where}
\quad\psi_{\varepsilon}(t)=\frac{1}{\varepsilon}\psi\left(\frac{t}{\varepsilon}\right),
\]
$\psi$ is a bounded variation function with support included in $[-1,1]$
and $\int_{\mathbb{R}}\psi(t)\, dt=1$. In addition, define 
\[
\tilde{Z}_{\alpha,\beta,\varepsilon}(t):=\varepsilon^{1-\alpha/2}\frac{d}{dt}B_{\alpha,
\beta}^{\varepsilon}(t),\quad t\geq0.
\]
We have the following weak convergence for the $\tilde{Z}_{\alpha,\beta,\varepsilon}$. 
\begin{thm}
1. For almost surely for all $x\in\mathbb{R}$ we have
\begin{equation}
\lambda\big\{ t\in[0,1],\;\tilde{Z}_{\alpha,\beta,\varepsilon}(t)\leq x\big\}\longrightarrow 
F_{C_{\psi}Y_{\beta}}(x),\;\varepsilon\rightarrow0,\label{eq:conv-law-reg}
\end{equation}
where $C_{\psi}$ is given by
\[
C_{\psi}=\left(-\frac{1}{2}\int_{-1}^{1}\int_{-1}^{1}|u-v|^{\alpha}d\psi(u)\, d\psi(v)\right)^{1/2}.
\]

2. For almost surely for each interval $I\subset\mathbb{R}_{+}$ and
all $x\in\mathbb{R}$, it follows from (\ref{eq:conv-law-reg}) that
\begin{equation}
\lambda\big\{ t\in I,\;\tilde{Z}_{\alpha,\beta,\varepsilon}(t)\leq x\big\}\longrightarrow
\lambda(I)F_{C_{\psi}Y_{\beta}}(x),\;\varepsilon\rightarrow0.\label{eq:conv-law-reg1}
\end{equation}

\end{thm}
We notice that  (\ref{eq:moduli-cont}) gives the $L^{p}$ moduli
of continuity of gBm. Next we state the $L^{p}$-moduli continuity
of squares of gBm which is based on Lemma~3.1 in \cite{Marcus-Rosen08}.
\begin{cor}
For any $p\in[1,\infty)$, we have the convergence of the $L^{p}$
moduli of the squares of gBm, as $\varepsilon\rightarrow0$ 
\begin{equation}
\int_{0}^{t}\varepsilon^{-\alpha/2}\big|B_{\alpha,\beta}^{2}(s+\varepsilon)-B_{\alpha,\beta}^{2}(s)
\big|^{p}\, ds\longrightarrow2^{p}Y_{\beta}^{p/2}\mathbb{E}\big(|N|^{p}\big)\int_{0}^{t}|B_{\alpha,
\beta}(s)|^{p}\, ds\label{eq:conv-p}
\end{equation}
holds for any $t\in[0,1]$ almost surely.
\end{cor}
To establish a similar convergence in law  for the increments of the
squares of gBm we need the convergence of the odd moments. This is
the contents of the following corollary.
\begin{cor}
1. For odd $k\in\mathbb{N}$ we have 
\begin{equation}
\lim_{\varepsilon\rightarrow0}\int_{0}^{t}\varepsilon^{-\alpha/2}\big(B_{\alpha,\beta}^{2}(s+\varepsilon)-
B_{\alpha,\beta}^{2}(s)\big)^{k}\, ds=0,\label{eq:conv-odd}
\end{equation}
almost surely for all $t\in[0,1]$. 

\noindent 2. Combining (\ref{eq:conv-p}) with even $p\in\mathbb{N}$
and (\ref{eq:conv-odd}) yields
\[
\lambda\big(\big\{ t\in[0,1]:\;\varepsilon^{-\alpha/2}\big(B_{\alpha,\beta}^{2}(t+\varepsilon)-
B_{\alpha,\beta}^{2}(t)\big)\leq x\big\}\big)\longrightarrow\int_{0}^{1}F_{4Y_{\beta}B_{\alpha,\beta}^{2}(s)}(x)\, ds.
\]
Here, the right hand side distribution, is a normal variance mixture.\end{cor}
\begin{proof}
For the proof of (\ref{eq:conv-odd}) see Appendix~\ref{sec:appendix}.
\end{proof}

\subsection{Weak convergence in the $w$-variable}

Using the results on stationary Gaussian process in \cite{Breuer_Major_1983},
\cite{Dobrushin-Major79}, \cite{Giraitis-Surgailis85}, \cite{Taqqu:1979}
and the independence of $Y_{\beta}$ and $B_{H}$ in the representation
(\ref{gbm-rep}) we obtain easily the following convergence results.
Below $\Delta B_{\alpha,\beta}(m/n)$ denotes the increment $B_{\alpha,\beta}((m+1)/n)-B_{\alpha,\beta}(m/n)$
for $m,n$ natural numbers.
\begin{thm}
We have the following limits:
\begin{enumerate}
\item If $k$ is even and $\alpha\in(0,\frac{3}{2})$, as $n\rightarrow\infty$,
\[
\frac{1}{\sqrt{n}}\sum_{m=0}^{n-1}[n^{k\alpha/2}\big(\Delta B_{\alpha,\beta}(m/n))^{k}-m_{k}(N)
Y_{\beta}^{k/2}\big)]\xrightarrow{\mathrm{Law}}Y_{\beta}^{k/2}N_{\sigma_{\alpha,k}^{2}}.
\]

\item If $k$ is even and $\alpha\in(\frac{3}{2},2)$, as $n\rightarrow\infty$,
\[
n^{1-\alpha}\sum_{m=0}^{n-1}[n^{k\alpha/2}(\Delta B_{\alpha,\beta}(m/n))^{k}-m_{k}(N)
Y_{\beta}^{k/2}]\xrightarrow{\mathrm{Law}}Y_{\beta}^{k/2}\mathrm{\; Rosenblatt\, r.v.}
\]

\item If $k$ is odd and $\alpha\in(0,1)$ , as $n\rightarrow\infty$,
\[
\frac{1}{\sqrt{n}}\sum_{m=0}^{n-1}[n^{k\alpha/2}(\Delta B_{\alpha,\beta}(m/n))^{k}
\xrightarrow{\mathrm{Law}}Y_{\beta}^{k/2}N_{\sigma_{\alpha,k}^{2}}.
\]

\item If $k$ is even and $\alpha\in(1,2)$, as $n\rightarrow\infty$,
\[
n^{-\alpha/2}\sum_{m=0}^{n-1}[n^{k\alpha/2}(\Delta B_{\alpha,\beta}(m/n))^{k}
\xrightarrow{\mathrm{Law}}Y_{\beta}^{k/2}N_{\sigma_{\alpha,k}^{2}}.
\]

\end{enumerate}
\end{thm}
On the above $m_{k}(N)$ denotes the $k$th moment of $N$ and $\sigma_{\alpha,k}$
is a positive constant depending only on $\alpha$ and $k$ which
may be different from one formula to another. The Rosenblatt random
variable is defined in \cite{Taqqu:1979}. 

In the same spirit of the above theorem, using Theorem~2.4 from \cite{Nourdin-2003}
we get the following weak convergence result. 
\begin{thm}
Let $k\geq3$ be an odd integer, then for $\alpha\in(0,1)$, we have,
as $\varepsilon\rightarrow0$ 
\[
\left\{ \frac{1}{\sqrt{\varepsilon}}\int_{0}^{t}Z_{\alpha,\beta,\varepsilon}^{k}(s)\, ds,\;0
\leq t\leq1\right\} \xrightarrow{\mathrm{Law}}\left\{ \sqrt{c_{k,\alpha}}Y_{\beta}^{k/2}B_{t},
\;0\leq t\leq1\right\} ,
\]
where $B_{t}$, $0\leq t\leq1$ is a standard Brownian motion and
$c_{k,\alpha}$ is given by 
\[
2\sum_{m=1}^{k}\frac{a_{k,m}^{2}}{k}\int_{0}^{\infty}\big[(x+1)^{\alpha}+(x-1)^{\alpha}-
2x^{\alpha}\big]^{m}dx
\]
and $a_{k,m}$ are the coefficient of $x^{k}$ in terms of the Hermite
polynomials, see \cite[pag.~6]{Nourdin-2003}.
\end{thm}

\section{Local times of gBm}

\label{sec:lt-gBm}

In this section we will prove the existence of gBm local times using
the criteria of Berman, \cite{Berman:1969wk}. More precisely, we
show that gBm possesses $\lambda$-square integrable local times,
cf.~Theorem~\ref{thm:gBm_LT} below. Moreover, this gBm local times
is weak-approximated by the number of crossings (also called level
sets) associated to the regularized gBm $B_{\alpha,\beta}^{\varepsilon}$,
this is established  in Theorem~\ref{thm:Level_sets_gBm_LT}.

\subsection{Existence of local times for gBm}

For the readers convenience we recall the notion of occupation measure
and occupation density. Let $f:I\longrightarrow\mathbb{R}$, $I$
a Borel set in $[0,1]$, be a measurable function and define, for
any set $B\in\mathcal{B}(\mathbb{R})$, the occupation measure $\mu_{f}$
on $I$ by 
\[
\mu_{f}(B):=\int_{I}1\!\!1_{B}(f(s))\, ds.
\]
Interpreting $[0,1]$ as a ``time set'' this is the ``amount of
time spent by $f$ in $B$ during the time period $I$''. We say
that $f$ has an occupation density over $I$ if $\mu_{f}$ is absolutely
continuous with respect to the Lebesgue measure $\lambda$ and denote
this density by $L^{f}(\cdot,I)$. In explicit, for any $x\in\mathbb{R}$,
\[
L^{f}(x,I)=\frac{d\mu_{f}}{d\lambda}(x).
\]
Thus, we have
\[
\mu_{f}(B)=\int_{I}1\!\!1_{B}(f(s))\, ds=\int_{B}L^{f}(x,I)\, dx.
\]
A continuous stochastic process $X$ has an occupation density on
$I$ if, for almost all $w\in\Omega$, $X(w)$ has an occupation density
$L^{X}(\cdot,I)$, also called local time of $X$, see Berman \cite{Berman:1969wk}. 

The criteria for the existence of local times for stochastic processes
are due to Berman \cite[Section~3]{Berman:1969wk}. More precisely,
a stochastic process $X$ admits a local times if and only if 
\begin{equation}
\int_{\mathbb{R}}\int_{0}^{1}\int_{0}^{1}\mathbb{E}\big(e^{i\theta(X(t)-X(s))}\big)\, ds\, dt\, 
d\theta<\infty.\label{Existence-LT}
\end{equation}
In the next we show that (\ref{Existence-LT}) is fulfilled if the
stochastic process $X$ is the gBm $B_{\alpha,\beta}$. In fact, from
(\ref{eq:cf_gBm_increments}) the characteristic function of the increments
of gBm $B_{\alpha,\beta}$ is given by

\begin{eqnarray*}
\mathbb{E}\big(e^{i\theta(B_{\alpha,\beta}(t)-B_{\alpha,\beta}(s))}\big) & = & 
E_{\beta}\left(-\frac{\theta^{2}}{2}|t-s|^{\alpha}\right).
\end{eqnarray*}
Using Fubini, and the change of variables $r=(2)^{-1/2}\theta|t-s|^{\alpha/2}$,
we have to compute at first 
\begin{eqnarray*}
\int_{\mathbb{R}}E_{\beta}\left(-\frac{\theta^{2}}{2}|t-s|^{\alpha}\right)\, d\theta & = & 
\frac{\sqrt{2}}{|t-s|^{\alpha/2}}\int_{\mathbb{R}}E_{\beta}(-r^{2})\, dr.
\end{eqnarray*}
The integral on the right and side is convergent, see \cite[eq.~(24)]{Berberan-Santos05}.
The $t,s$-integration is performed as follows.
\begin{eqnarray*}
\int_{0}^{1}\int_{0}^{1}\frac{1}{|t-s|^{\alpha/2}}\, ds\, dt & = & 2\int_{0}^{1}\int_{0}^{t}\frac{1}
{|t-s|^{\alpha/2}}\, ds\, dt\\
 & = & \frac{4}{2-\alpha}\int_{0}^{1}t^{1-\alpha/2}\, dt\\
 & = & \frac{8}{(2-\alpha)(4-\alpha)}.
\end{eqnarray*}

Thus, we have shown the main result of this subsection which we state
in the following theorem.
\begin{thm}
\label{thm:gBm_LT}The gBm process $B_{\alpha,\beta}$ admits a $\lambda$-square
integrable local time $L^{B_{\alpha,\beta}}(\cdot,I)$ almost surely.\end{thm}
\begin{rem}
As a consequence of the existence of the local time $L^{B_{\alpha,\beta}}(\cdot,I)$,
we obtain the occupation formula
\[
\int_{I}f(B_{\alpha,\beta}(s))\, ds=\int_{\mathbb{R}}f(x)L^{B_{\alpha.\beta}}(x,I)\, dx,\; a.s.
\]

\end{rem}

\subsection{Weak-approximation of gBm local times}

We now show that the gBm local times $L^{B_{\alpha,\beta}}(\cdot,I)$
is the weak-limit, $\varepsilon\rightarrow0$, of the set of number
of crossings corresponding to the regularized gBm $B_{\alpha,\beta}^{\varepsilon}$.
At first we recall the definition of the number of crossings of $B_{\alpha,\beta}^{\varepsilon}$.
The number of crossings of level $x$ of $B_{\alpha,\beta}^{\varepsilon}$
in the interval $I$ is defined by
\[
C^{B_{\alpha,\beta}^{\varepsilon}}(x,I):=\#\big\{ t\in I,\; B_{\alpha,\beta}^{\varepsilon}(t)=x\big\}.
\]

\begin{thm}
\label{thm:Level_sets_gBm_LT}For any continuous bounded real function
$f$ and any bounded interval $I$, almost surely, we have 
\[
\varepsilon^{1-\alpha/2}\sqrt{\frac{\pi}{2}}C_{\psi}^{-1}\int_{\mathbb{R}}f(x)C^{B_{\alpha,\beta}^
{\varepsilon}}(x,I)\, dx\xrightarrow[\varepsilon\rightarrow0]{}\sqrt{Y_{\beta}}\int_{\mathbb{R}}f(x)
L^{B_{\alpha,\beta}}(x,I)\, dx.
\]
\end{thm}
\begin{proof}
For any continuous bounded function $f$ we have
\begin{eqnarray*}
 &  & \varepsilon^{1-\alpha/2}\int_{\mathbb{R}}f(x)C^{B_{\alpha,\beta}^
 {\varepsilon}}(x,I)\, dx\\
 & = & \varepsilon^{1-\alpha/2}\int_{I}f(B_{\alpha,\beta}^{\varepsilon}(t))
 \left|\frac{d}{dt}B_{\alpha,\beta}^{\varepsilon}(t)\right|\, dt\\
 & = & \int_{I}\big(f(B_{\alpha,\beta}^{\varepsilon}(t))-f(B_{\alpha,\beta}(t))\big)
 \left|\varepsilon^{1-\alpha/2}\frac{d}{dt}B_{\alpha,\beta}^{\varepsilon}(t)\right|\, dt\\
 &  & +\int_{I}f(B_{\alpha,\beta}(t))\left|\varepsilon^{1-\alpha/2}\frac{d}{dt}
 B_{\alpha,\beta}^{\varepsilon}(t)\right|\, dt.
\end{eqnarray*}
Since $B_{\alpha,\beta}$ and $f$ are continuous it follows that,
almost surely 
\[
\lim_{\varepsilon\rightarrow0}\sup_{t\in I}\big|f(B_{\alpha,\beta}^{\varepsilon}(t))-
f(B_{\alpha,\beta}(t))\big|=0
\]
and 
\[
\sup_{\varepsilon>0}\int_{I}\left|\varepsilon^{1-\alpha/2}\frac{d}{dt}B_{\alpha,\beta}^
{\varepsilon}(t)\right|\, dt<\infty.
\]
So we have
\[
\lim_{\varepsilon\rightarrow0}\int_{I}\big[f(B_{\alpha,\beta}^{\varepsilon}(t))-
f(B_{\alpha,\beta}(t))\big]\left|\varepsilon^{1-\alpha/2}\frac{d}{dt}B_{\alpha,\beta}^
{\varepsilon}(t)\right|\, dt=0.
\]
For the other integral, we obtain from (\ref{eq:moduli-cont}), with
$p=1$, that
\[
\lim_{\varepsilon\rightarrow0}\int_{I}f(B_{\alpha,\beta}(t))\left|\varepsilon^{1-
\alpha/2}\frac{d}{dt}B_{\alpha,\beta}^{\varepsilon}(t)\right|\, dt=\sqrt{\frac{2}{\pi}}
C_{\psi}\sqrt{Y_{\beta}}\int_{I}f(B_{\alpha,\beta}(t))\, dt.
\]
The result of the theorem follows from the occupation formula. 
\end{proof}
\appendix

\section{\label{sec:appendix}Proof of (\ref{eq:conv-odd})}

Having in mind the representation (\ref{gbm-rep}) it is enough to
show that almost surely, for any $t\in[0,1]$, we have 
\[
\lim_{\varepsilon\rightarrow0}\int_{0}^{t}\left(\frac{B_{H}^{2}(s+\varepsilon)-
B_{H}^{2}(s)}{\varepsilon^{H}}\right)^{2k+1}ds=0.
\]
It is sufficient to show that 
\begin{equation}
\lim_{\varepsilon\rightarrow0}\int_{0}^{t}\left(\frac{B_{H}(s+\varepsilon)-
B_{H}(s)}{\varepsilon^{H}}\right)^{2k+1}B_{H}^{2k+1}(s)\, ds=0\label{eq:conv-frac}
\end{equation}
since the Hölder continuity of fBm implies 
\[
\lim_{\varepsilon\rightarrow0}\sup_{s\in[0,1]}|B_{H}(s+\varepsilon)-B_{H}(s)|=0.
\]
Let $B_{n}$, $n\in\mathbb{N}$ be a $C^{1}$ regularizing sequence
of $B_{H}$. We have
\begin{equation}
\lim_{\varepsilon\rightarrow0}\int_{0}^{t}\left(\frac{B_{H}(s+\varepsilon)-
B_{H}(s)}{\varepsilon^{H}}\right)^{2k+1}B_{n}^{2k+1}(s)\, ds=0.\label{eq:conv-reg}
\end{equation}
In fact, using the integration by parts we obtain
\begin{eqnarray*}
 &  & \int_{0}^{t}\left(\frac{B_{H}(s+\varepsilon)-B_{H}(s)}{\varepsilon^{H}}
 \right)^{2k+1}B_{n}^{2k+1}(s)\, ds\\
 & = & B_{n}^{2k+1}(t)\int_{0}^{t}\left(\frac{B_{H}(s+\varepsilon)-B_{H}(s)}
 {\varepsilon^{H}}\right)^{2k+1}ds\\
 &  & -\int_{0}^{t}\frac{d}{ds}B_{n}^{2k+1}(s)\int_{0}^{s}\left(\frac{B_{H}
 (u+\varepsilon)-B_{H}(u)}{\varepsilon^{H}}\right)^{2k+1}du\, ds.
\end{eqnarray*}
Now (\ref{eq:conv-reg}) follows from the fact that almost surely
uniformly in $t\in[0,1]$ we have (cf.~Theorem~2.1 \cite{Nourdin-2003})
\[
\lim_{\varepsilon\rightarrow0}\int_{0}^{t}\left(\frac{B_{H}(s+\varepsilon)-
B_{H}(s)}{\varepsilon^{H}}\right)^{2k+1}ds=0.
\]
Finally (\ref{eq:conv-frac}) is a consequence of uniform convergence
on $[0,1]$ of the regularizing sequence $B_{n}$ to $B_{H}$. This
finish the prove.

\subsection*{Acknowledgments}

We would like to thank our colleague and friend Y.~Ouknine for the
hospitality and friendship during various visits to Marrakech where
most of this work was done. Financial support of the project CCM -
PEst-OE/MAT/UI0219/2011 and Marie Curie ITN grant FP7-PEOPLE-2007-1-1-ITN,
no.~213841 are gratefully acknowledged.


\begin{thebibliography}{MMP10}

\bibitem[AW96]{Azais:1996}
J.~M. Aza{\"\i}s and M.~Wschebor.
\newblock {Almost sure oscillation of certain random processes}.
\newblock {\em Bernoulli}, 2(3):257--270, 1996.

\bibitem[Ber69]{Berman:1969wk}
S.~M. Berman.
\newblock {Local times and sample function properties of stationary Gaussian
  processes}.
\newblock {\em Transactions of the American Mathematical Society},
  137:277--299, 1969.

\bibitem[Bin71]{Bingham1971}
N.~H. Bingham.
\newblock Limit theorems for occupation times of {M}arkov processes.
\newblock {\em Z.~Wahrsch.~verw.~{G}ebiete}, 17:1--22, 1971.

\bibitem[BKS96]{Bondesson1996}
L.~Bondesson, G.~K. Kristiansen, and F.~W. Steutel.
\newblock Infinite divisibility of random variables and their integer parts.
\newblock {\em Statist.~Probab.~Lett.}, 28:271--278, 1996.

\bibitem[BM83]{Breuer_Major_1983}
P.~Breuer and P.~Major.
\newblock Central limit theorems for nonlinear functionals of {G}aussian
  fields.
\newblock {\em J.~Multivariate Anal.}, 13(3):425--441, 1983.

\bibitem[BS05]{Berberan-Santos05}
M.~N. Berberan-Santos.
\newblock Properties of the {M}ittag-{L}effler relaxation function.
\newblock {\em J.~Math.~Chem.}, 38(4):629--635, 2005.

\bibitem[DM79]{Dobrushin-Major79}
R.~L. Dobrushin and P.~Major.
\newblock {Non-central limit theorems for non-linear functional of Gaussian
  fields}.
\newblock {\em Z.~Wahrsch.~verw.~{G}ebiete}, 50(1):27--52, 1979.

\bibitem[FKN90]{Fang90}
K.~T. Fang, S.~Kotz, and K.~W. Ng.
\newblock {\em Symmetric multivariate and related distributions}, volume~36 of
  {\em Monographs on Statistics and Applied Probability}.
\newblock Chapman and Hall Ltd., London, 1990.

\bibitem[GN03]{Nourdin-2003}
M.~Gradinaru and I.~Nourdin.
\newblock Approximation at first and second order of {$m$}-order integrals of
  the fractional {B}rownian motion and of certain semimartingales.
\newblock {\em Electron.~J.~Probab.}, 8:1--26, 2003.

\bibitem[GS85]{Giraitis-Surgailis85}
L~Giraitis and D.~Surgailis.
\newblock {CLT and other limit theorems for functionals of Gaussian processes}.
\newblock {\em Z.~Wahrsch.~verw.~Gebiete}, 70(2):191--212, 1985.

\bibitem[Kra06]{Kratz2006}
M.~F. Kratz.
\newblock Level crossings and other level functionals of stationary {G}aussian
  processes.
\newblock {\em Prob.~Surv.}, 3:230--288, December 2006.

\bibitem[MM09]{Mura_mainardi_09}
A.~Mura and F.~Mainardi.
\newblock A class of self-similar stochastic processes with stationary
  increments to model anomalous diffusion in physics.
\newblock {\em Integral Transforms Spec.~Funct.}, 20(3-4):185--198, 2009.

\bibitem[MMP10]{Mainardi_Mura_Pagnini_2010}
F.~Mainardi, A.~Mura, and G.~Pagnini.
\newblock The {$M$}-{W}right function in time-fractional diffusion processes: A
  tutorial survey.
\newblock {\em Int.~J.~Differ.~Equ.}, pages Art.~ID 104505, 29, 2010.

\bibitem[MP08]{Mura_Pagnini_08}
A.~Mura and G.~Pagnini.
\newblock Characterizations and simulations of a class of stochastic processes
  to model anomalous diffusion.
\newblock {\em J.~Phys.~A}, 41(28):285003, 22, 2008.

\bibitem[MR08]{Marcus-Rosen08}
M.~B. Marcus and J.~Rosen.
\newblock {$L^p$} moduli of continuity of {G}aussian processes and local times
  of symmetric {L}{\'e}vy processes.
\newblock {\em Ann.~Probab.}, 36(2):594--622, 2008.

\bibitem[Sch90]{Schneider90}
W.~R. Schneider.
\newblock Grey noise.
\newblock In S.~Albeverio~et al., editor, {\em Stochastic processes, physics
  and geometry}, pages 676--681. World Sci.~Publ., Teaneck, NJ, 1990.

\bibitem[Sch92]{MR1190506}
W.~R. Schneider.
\newblock Grey noise.
\newblock In {\em Ideas and methods in mathematical analysis, stochastics, and
  applications ({O}slo, 1988)}, pages 261--282. Cambridge Univ.~Press,
  Cambridge, 1992.

\bibitem[Shi96]{S96}
A.~N. Shiryaev.
\newblock {\em Probability}.
\newblock Springer Verlag, New York Berlin Heidelgerg, 1996.

\bibitem[SvH04]{Steutel_van_Harn_2004}
F.~W. Steutel and K.~van Harn.
\newblock {\em Infinite divisibility of probability distributions on the real
  line}, volume 259 of {\em Monographs and Textbooks in Pure and Applied
  Mathematics}.
\newblock Marcel Dekker Inc., New York, 2004.

\bibitem[Taq79]{Taqqu:1979}
M.~S. Taqqu.
\newblock {Convergence of integrated processes of arbitrary Hermite rank}.
\newblock {\em Z.~Wahrsch.~verw.~Gebiete}, 50(1):53--83, 1979.

\end{thebibliography}
\end{document}